\documentclass[12pt]{article}
\usepackage[utf8]{inputenc}

\usepackage{amsmath,mathtools,amsthm,amssymb, xcolor}
\usepackage{thm-restate}
\usepackage{tikz}
\usepackage{enumitem}
\usetikzlibrary{calc}
\usepackage{fullpage}
\usepackage{makecell}
\usepackage{stmaryrd}

\usepackage[numbers]{natbib}
\newcommand{\vc}[1]{\ensuremath{\vcenter{\hbox{#1}}}}
\tikzset{flag_pic/.style={scale=1}}  %  Change the scale to change all figures
\tikzset{unlabeled_vertex/.style={inner sep=1.7pt, outer sep=0pt, circle, fill}}
\tikzset{labeled_vertex/.style={inner sep=3pt, outer sep=0pt, rectangle, fill=white, draw=black}}
\tikzset{edge_color0/.style={color=black,line width=1.2pt,opacity=0.5,dashed}}
\tikzset{edge_color1/.style={color=red,  line width=1.2pt,opacity=0}} 
\tikzset{edge_color2/.style={color=blue, line width=1.2pt,opacity=1, dashed}}
\tikzset{edge_color3/.style={color=green!80!black,line width=1.2pt}}
\tikzset{edge_color4/.style={color=orange, line width=1.2pt}}
\tikzset{edge_color5/.style={color=red,  line width=1.2pt,dotted}}
\tikzset{edge_color6/.style={color=blue, line width=1.2pt,dotted}}
\tikzset{edge_color7/.style={color=green, line width=1.2pt,dotted}}
\tikzset{edge_color8/.style={color=gray, line width=1.2pt}}
\tikzset{edge_color9/.style={color=gray, dotted, line width=1.2pt}}
\tikzset{edge_color10/.style={color=gray, dashed, line width=1.2pt}}
\tikzset{edge_color11/.style={color=pink, dashed, line width=1.2pt}}
\tikzset{edge_f/.style={-latex}}
\tikzset{edge_b/.style={latex-}}
\tikzset{edge_color1f/.style={edge_f,edge_color1}}
\tikzset{edge_color1b/.style={edge_b,edge_color1}}
\tikzset{edge_color2f/.style={edge_f,edge_color2}}
\tikzset{edge_color2b/.style={edge_b,edge_color2}}
\tikzset{edge_color3f/.style={edge_f,edge_color3}}
\tikzset{edge_color3b/.style={edge_b,edge_color3}}
\tikzset{edge_color4f/.style={edge_f,edge_color4}}
\tikzset{edge_color4b/.style={edge_b,edge_color4}}
\tikzset{edge_color5f/.style={edge_f,edge_color5}}
\tikzset{edge_color5b/.style={edge_b,edge_color5}}
\tikzset{edge_color6f/.style={edge_f,edge_color6}}
\tikzset{edge_color6b/.style={edge_b,edge_color6}}
\tikzset{edge_color7f/.style={edge_f,edge_color7}}
\tikzset{edge_color7b/.style={edge_b,edge_color7}}
\tikzset{edge_color8f/.style={edge_f,edge_color8}}
\tikzset{edge_color8b/.style={edge_b,edge_color8}}
\tikzset{edge_colorroot/.style={color=red, line width=1.7pt}}
\tikzset{edge_thin/.style={color=black}}
\tikzset{edge_hidden/.style={color=black,dotted,opacity=0}}
\tikzset{vertex/.style={inner sep=1.7pt, outer sep=0pt, circle}}
\tikzset{vertex_color0/.style={inner sep=1.7pt, outer sep=0pt, draw, circle, fill=white}}
\tikzset{vertex_color1/.style={inner sep=1.7pt, outer sep=0pt, draw, circle, fill=red!30!white}}
\tikzset{vertex_color2/.style={inner sep=1.7pt, outer sep=0pt, draw, circle, fill=blue!30!white}}
\tikzset{vertex_color3/.style={inner sep=1.7pt, outer sep=0pt, draw, circle, fill=green}}
\tikzset{vertex_color4/.style={inner sep=1.7pt, outer sep=0pt, draw, circle, fill=yellow}}
\tikzset{vertex_color5/.style={inner sep=1.7pt, outer sep=0pt, draw, circle, fill=gray!30!white}}
\tikzset{vertex_color6/.style={inner sep=1.7pt, outer sep=0pt, draw, circle, fill=gray,label=below:{$6$}}}
\tikzset{vertex_color7/.style={inner sep=1.7pt, outer sep=0pt, draw, circle, fill=gray,label=below:{$7$}}}
\tikzset{vertex_color8/.style={inner sep=1.7pt, outer sep=0pt, draw, circle, fill=gray,label=below:{$8$}}}
\tikzset{vertex_color9/.style={inner sep=1.7pt, outer sep=0pt, draw, circle, fill=gray,label=below:{$9$}}}
\tikzset{vertex_color10/.style={inner sep=1.7pt, outer sep=0pt, draw, circle, fill=gray,label=below:{$10$}}}
\tikzset{vertex_color11/.style={inner sep=1.7pt, outer sep=0pt, draw, circle, fill=gray,label=below:{$11$}}}
\tikzset{vertex_color12/.style={inner sep=1.7pt, outer sep=0pt, draw, circle, fill=gray,label=below:{$12$}}}
\tikzset{vertex_color13/.style={inner sep=1.7pt, outer sep=0pt, draw, circle, fill=gray,label=below:{$13$}}}
\tikzset{vertex_color14/.style={inner sep=1.7pt, outer sep=0pt, draw, circle, fill=gray,label=below:{$14$}}}
\tikzset{labeled_vertex_color0/.style={inner sep=3pt, outer sep=0pt, draw, rectangle, fill=white}}
\tikzset{labeled_vertex_color1/.style={inner sep=3pt, outer sep=0pt, draw, rectangle, fill=red!30!white}}
\tikzset{labeled_vertex_color2/.style={inner sep=3pt, outer sep=0pt, draw, rectangle, fill=blue!30!white}}
\tikzset{labeled_vertex_color3/.style={inner sep=3pt, outer sep=0pt, draw, rectangle, fill=green}}
\tikzset{labeled_vertex_color4/.style={inner sep=3pt, outer sep=0pt, draw, rectangle, fill=yellow}}
\tikzset{labeled_vertex_color5/.style={inner sep=3pt, outer sep=0pt, draw, rectangle, fill=gray!30!white}}
\tikzset{labeled_vertex_color6/.style={inner sep=3pt, outer sep=0pt, draw, rectangle, fill=gray,label=below:{$6$}}}
\tikzset{labeled_vertex_color7/.style={inner sep=3pt, outer sep=0pt, draw, rectangle, fill=gray,label=below:{$7$}}}
\tikzset{labeled_vertex_color8/.style={inner sep=3pt, outer sep=0pt, draw, rectangle, fill=gray,label=below:{$8$}}}
\tikzset{labeled_vertex_color9/.style={inner sep=3pt, outer sep=0pt, draw, rectangle, fill=gray,label=below:{$9$}}}
\tikzset{text_color0/.style={color=black}}
\tikzset{text_color1/.style={color=red}}
\tikzset{text_color2/.style={color=blue}}
\tikzset{text_color3/.style={color=green!70!black}}
\tikzset{text_color4/.style={color=orange}}
\tikzset{text_color5/.style={color=gray}}

\def\outercycle#1#2{
\pgfmathtruncatemacro{\plusone}{#1+1}
\pgfmathtruncatemacro{\zeroshift}{270 - (#2-1)*360/#1/2 }
    \draw  \foreach \x in {0,1,...,#1}{(\zeroshift+\x*360/#1:1) node[vertex](x\x){}};
}

\def\labelvertex#1{\pgfmathtruncatemacro{\vertexlabel}{#1+1 } \draw (x#1) node{\color{black}\tiny\vertexlabel}; }
\usepackage{listofitems}
\tikzset{vertex_u/.style={unlabeled_vertex}}
\tikzset{vertex_l/.style={labeled_vertex}}
\newcommand{\Fnv}[2]{ %   #of labeled vertices  #vertex_ID   draws the vertex and adds label if applicable
\ifnum#2<#1 \draw (x#2) node[vertex_l]{}; \labelvertex{#2}  % labeled vertex
\else  \draw (x#2) node[vertex_u]{}; \fi % unlabeled vertex
}
\newcommand{\Fne}[3]{%   edge_color   i   j     draws an edge of edge_color between xi and  xj
\draw[edge_color#3] (x#1)--(x#2); 
}
\newcounter{Fneid} % tmp counter
\newcommand{\Fe}[3]{%            #vertices   #labeled   {color color color ...  } adjacency matrix
\ifnum#1=1% If we have just one vertex, draw it and don't worry about the rest
  \vc{\begin{tikzpicture}[scale=0.4]\outercycle{1}{2}
  \Fnv{#2}{0}
  \end{tikzpicture}}
\else
\setsepchar{ }% edge labels are separated by spaces
\readlist\elabel{#3}% parse edge labels to \elabel  
\pgfmathtruncatemacro{\vertexloop}{#1-1}% when looping vertices, we need -1 
\pgfmathtruncatemacro{\vertexloopi}{#1-2}% and -2 in the loop
\pgfmathtruncatemacro{\expectededges}{#1*(#1-1)/2}%   how many edges we should have?
\ifnum\elabellen=\expectededges% 
\def\cycleshift{2}% Check how to rotate the image, default is 2, but sometimes we can change it
\ifnum#1=2\def\cycleshift{1}\fi%          for 2 vertices it is better to shift,
\ifnum#1=3\ifnum#2=1\def\cycleshift{1}\fi\fi%   for 3 vertices with 1 labeled we rotate them
\def\Fnscale{0.4}\ifnum#1>4\def\Fnscale{0.45}\fi\ifnum#1>5\def\Fnscale{0.55}\fi\ifnum#1>7\def\Fnscale{0.65}\fi% scale the figure for larger graphs to make it readable
\vc{\begin{tikzpicture}[scale=\Fnscale]
          \outercycle{#1}{\cycleshift}% Define coordinates of vertices 
          \setcounter{Fneid}{1}         % we need to keep track of index of edge in  
          \foreach\i in {0,...,\vertexloopi}{   % loop over \i, \j pairs of vertices
          \pgfmathtruncatemacro{\jfrom}{\i+1}
          \foreach\j in {\jfrom,...,\vertexloop}{
          \edef\eID{\arabic{Fneid}}
          \edef\eij{\elabel[\eID]}         % get the label of the edge i,j
            \Fne\i\j{\eij}              % draw the edge i,j
	    \stepcounter{Fneid}
          }}
          \foreach\i in {0,...,\vertexloop}{\Fnv{#2}{\i}  % draw all vertices
          }
          \end{tikzpicture}}%
\else% If the number of edges does not match, print an error message
   #1 vertices need \expectededges{} edges but got \elabellen edges.
\fi
\fi
}
\usepackage{listofitems}
\tikzset{vertex_u/.style={unlabeled_vertex}}
\tikzset{vertex_l/.style={labeled_vertex}}
\newcommand{\FnvD}[3]{ %   #of labeled vertices  #vertex_ID   draws the vertex and adds label if applicable
\ifnum#2<#1 \draw (x#2) node[labeled_vertex_color#3]{}; \labelvertex{#2}  % labeled vertex
\else  \draw (x#2) node[vertex_color#3]{}; \fi % unlabeled vertex
}

\newcounter{Fvneid} % tmp counter
\newcounter{Fvnvid} % tmp counter

\newcommand{\Fve}[4]{%            #vertices   #labeled   {color color color ...  } adjacency matrix
\ifnum#1=1% If we have just one vertex, draw it and don't worry about the rest
  \vc{\begin{tikzpicture}[scale=0.4]\outercycle{1}{2}
  \FnvD{#2}{0}{#3}
  \end{tikzpicture}}
\else
\setsepchar{ }% edge labels are separated by spaces
\readlist\vlabel{#3}% parse edge labels to \vlabel  
\readlist\elabel{#4}% parse edge labels to \elabel  
\pgfmathtruncatemacro{\vertexloop}{#1-1}% when looping vertices, we need -1 
\pgfmathtruncatemacro{\vertexloopi}{#1-2}% and -2 in the loop
\pgfmathtruncatemacro{\expectededges}{#1*(#1-1)/2}%   how many edges we should have?
\ifnum\vlabellen=#1% 
\ifnum\elabellen=\expectededges% 
\def\cycleshift{2}% Check how to rotate the image, default is 2, but sometimes we can change it
\ifnum#1=2\def\cycleshift{1}\fi%          for 2 vertices it is better to shift,
\ifnum#1=3\ifnum#2=1\def\cycleshift{1}\fi\fi%   for 3 vertices with 1 labeled we rotate them
\def\Fnscale{0.4}\ifnum#1>4\def\Fnscale{0.45}\fi\ifnum#1>5\def\Fnscale{0.55}\fi\ifnum#1>7\def\Fnscale{0.65}\fi% scale the figure for larger graphs to make it readable
\vc{\begin{tikzpicture}[scale=\Fnscale]
          \outercycle{#1}{\cycleshift}% Define coordinates of vertices 
          \setcounter{Fvneid}{1}         % we need to keep track of index of edge in  
          \foreach\i in {0,...,\vertexloopi}{   % loop over \i, \j pairs of vertices
          \pgfmathtruncatemacro{\jfrom}{\i+1}
          \foreach\j in {\jfrom,...,\vertexloop}{
          \edef\eID{\arabic{Fvneid}}%
          \edef\eij{\elabel[\eID]} % get the label of the edge i,j
            \Fne\i\j{\eij}%  draw the edge i,j
	    \stepcounter{Fvneid}
          }}
          \foreach\i in {0,...,\vertexloop}{%
          \setcounter{Fvnvid}{\i}
          \stepcounter{Fvnvid}
          \edef\vID{\arabic{Fvnvid}}%
          \edef\vlabeli{\vlabel[\vID]}%
          \FnvD{#2}{\i}{\vlabeli} % draw all vertices
          }%
          \end{tikzpicture}}%
\else% If the number of edges does not match, print an error message
   \text{ #1 vertices need \expectededges{} edges but got \elabellen{} edges. }
\fi
\else% If the number of edges does not match, print an error message
 \text{ #1 vertices need #1 vertex labels but got \vlabellen{} labels. }
\fi
\fi
}
\usepackage{ifthen}

\tikzset{vertex_u/.style={unlabeled_vertex}}
\tikzset{vertex_l/.style={labeled_vertex}}

\newcommand{\Fuu}[1]{
\,\vc{\begin{tikzpicture}[scale=0.3]\outercycle{2}{1}
\draw[edge_color#1] (x0)--(x1);
\draw (x0) node[unlabeled_vertex]{};\draw (x1) node[unlabeled_vertex]{};
\end{tikzpicture}}
\,
}

\newcommand{\Fluu}[3]{
\vc{\begin{tikzpicture}[scale=0.4]\outercycle{3}{1}
\draw[edge_color#1] (x0)--(x1);\draw[edge_color#2] (x0)--(x2);  \draw[edge_color#3] (x1)--(x2);
\draw (x0) node[labeled_vertex]{};\draw (x1) node[unlabeled_vertex]{};\draw (x2) node[unlabeled_vertex]{};
\labelvertex0
\end{tikzpicture}}}

\newcommand{\Flll}[3]{
\vc{\begin{tikzpicture}[scale=0.4]\outercycle{3}{2}
\draw[edge_color#1] (x0)--(x1);\draw[edge_color#2] (x0)--(x2);  \draw[edge_color#3] (x1)--(x2);
\draw (x0) node[labeled_vertex]{};\draw (x1) node[labeled_vertex]{};\draw (x2) node[labeled_vertex]{};
\labelvertex0
\labelvertex1
\labelvertex2
\end{tikzpicture}}}

\newcommand{\FfourEdges}[6]{
\draw[edge_color#1] (x0)--(x1);\draw[edge_color#2] (x0)--(x2);\draw[edge_color#3] (x0)--(x3);  \draw[edge_color#4] (x1)--(x2);\draw[edge_color#5] (x1)--(x3);  \draw[edge_color#6] (x2)--(x3);
}
\newcommand{\Ffour}[5]{
\vc{\begin{tikzpicture}[scale=0.4]\outercycle{4}{2}
\FfourEdges#5
\draw (x0) node[vertex_#1]{};\draw (x1) node[vertex_#2]{};\draw (x2) node[vertex_#3]{};\draw (x3) node[vertex_#4]{};
\ifthenelse{\equal{#1}{l}}{\labelvertex{0}}{}%
\ifthenelse{\equal{#2}{l}}{\labelvertex{1}}{}%
\ifthenelse{\equal{#3}{l}}{\labelvertex{2}}{}%
\ifthenelse{\equal{#4}{l}}{\labelvertex{3}}{}%
\end{tikzpicture}}
}

\newcommand{\Flluu}[6]{\Ffour{l}{l}{u}{u}{#1#2#3#4#5#6}}

\usepackage[pdftex,hypertexnames=false,linktocpage=true]{hyperref}
\hypersetup{colorlinks=true,linkcolor=black,anchorcolor=blue,citecolor=black,filecolor=blue,urlcolor=black,bookmarksnumbered=true,pdfview=FitB}
\usepackage{cleveref}

\newtheorem{theorem}{Theorem}[section]
\newtheorem{corollary}[theorem]{Corollary}
\newtheorem{lemma}[theorem]{Lemma}

\newtheorem{proposition}[theorem]{Proposition}
\newtheorem{problem}[theorem]{Problem}
\newtheorem{question}[theorem]{Question}
\newtheorem{conjecture}[theorem]{Conjecture}
\newtheorem{claim}[theorem]{Claim}

\theoremstyle{definition}
\newtheorem{definition}{Definition}[section]

\newcommand{\oururl}{\url{https://lidicky.name/pub/rtcuts/}}

\tikzset{
vtx/.style={inner sep=1.7pt, outer sep=0pt, circle, fill=black,draw=black}, % style of a vertex
}

\tikzset{
xvtx/.style={inner sep=1.7pt, outer sep=0pt, circle,draw=black}, % style of a vertex
}

\title{Bipartite cuts in Ramsey-Tur\'an style}

\author{
{J\'ozsef Balogh\thanks{Department of Mathematics, University of Illinois Urbana-Champaign, Urbana, IL, USA, and Extremal Combinatorics and Probability Group (ECOPRO), Institute for Basic Science (IBS), Daejeon, South Korea. Email: \texttt{jobal@illinois.edu}. Partially supported by NSF grants RTG DMS-1937241, FRG DMS-2152488, UIUC  Campus Research Board Award RB26026, the  Simons Collaboration grant [SFI-MPS-TSM-00013107, JB], and the Institute for Basic Science (IBS-R029-C4).}
}
\and
Ce Chen\thanks{
Department of Mathematics, University of Illinois Urbana-Champaign, Urbana, IL, USA. E-mail: {\tt cechen4@illinois.edu}. Partially supported by NSF grants RTG DMS-1937241,  FRG DMS-2152488, and UIUC Campus Research Board RB 24012.
}
\and
Bernard Lidick\'y\thanks{
Department of Mathematics, Iowa State University, Ames, IA, USA. E-mail: {\tt lidicky@iastate.edu}. Research of this author is supported in part by NSF FRG DMS-2152490, the Simons Foundation TSM-00013439 and Scott Hanna Professorship.
}
}

\date{}

\begin{document}

\maketitle

\begin{abstract}
We prove that every $K_5$-free $n$-vertex graph with sublinear independence number can be made bipartite by removing at most $n^2(1/18+o(1))$ edges, where the constant $1/18$ is best possible.
The proof method is related to extensions of Tur\'an Theorem in edge-weighted settings, and part of the proof uses flag algebra.

Using a similar flag algebra formulation, AI proved that every
$K_6$-free and $K_7$-free $n$-vertex graph with sublinear independence number can be made bipartite by removing at most $n^2(1/12+o(1))$ and $n^2(1/9+o(1))$ edges, respectively. 
These constants are best possible.
\end{abstract}


\section{Introduction}\label{sec: intro}
How many edges are needed to be removed from a triangle-free graph on $n$ vertices to make it bipartite?
Erd\H{o}s~\cite{MR0409246} asked this question and conjectured that $n^2/25$ edges would always be sufficient. This would be best possible as the balanced blow-up of $C_5$ with class sizes $n/5$ needs at least $n^2/25$ edges to be removed to make it bipartite.
For a graph $H$ and an integer $n$, denote by $D_2^\#(n, H)$ 
the minimum integer $d$ such that every $n$-vertex $H$-free graph can be made bipartite by removing at most $d$ edges.
%the minimum number of edges which have to be removed to make any $n$-vertex $H$-free graph bipartite. 
Define
\[
D_2(H)\coloneqq\limsup_{n\to\infty}\frac{D_2^\#(n, H)}{\binom{n}{2}}.
\]
The problem of determining $D_2(H)$ and its variants has attracted considerable attention in recent years.
For $H=K_3$, a closely related topic is Erd\H{o}s's Sparse Half Conjecture~\cite{MR0409246,erdHos1997some}.
Razborov~\cite{Razborov2022} completely proved the conjecture for graphs of girth at least $5$, for graphs with independence number at least $2n/5$ and for strongly regular graphs.
Balogh, Clemen, and Lidick\'y~\cite{Balogh2024} systematically investigated several related partition problems for triangle-free graphs and proposed a number of conjectures and open problems. 
For the analogous setting of $K_4$-free graphs, Liu and Ma~\cite{Liu2021} proved the existence of sparse halves for regular $K_4$-free graphs, and Reiher~\cite{Reiher2022} later settled the problem for every $K_4$-free graph.
Very recently, Balogh, Buczek, Grzesik, and Kuc~\cite{balogh2026} solved a balanced  bipartite variant of the problem for $K_4$-free graphs.

Solving a conjecture of Erd\H os~\cite{erdos1988makebi}, Sudakov~\cite{MR2359832} proved the following theorem.
\begin{theorem}[Sudakov~\cite{MR2359832}]\label{thm:sudakov}
Every $n$-vertex $K_4$-free graph can be made bipartite by removing at most $n^2/9$ edges, i.e., $D_2^\#(K_4)\leq n^2/9$. 
\end{theorem}
This theorem is best possible since from the balanced complete 3-partite graph on $n$ vertices, at least $n^2/9$ edges are needed to be removed to make it bipartite.
Sudakov~\cite{MR2359832} also conjectured a generalization of his result, stating that the balanced complete $(r-1)$-partite graph on $n$ vertices is the worst graph for every $r\geq 5$.
\begin{conjecture}[Sudakov~\cite{MR2359832}]\label{sudakov}
    For all $r \geq 5$
    \[
D_2^\#(K_r) = \begin{cases}
  \frac{r-3}{4(r-1)}n^2  & \quad\quad\quad \text{for odd } r,\\[5pt] 
  \frac{(r-2)^2}{4(r-1)^2}n^2  & \quad\quad\quad \text{for even } r.\\  
\end{cases}
    \]
\end{conjecture}

In particular, Conjecture~\ref{sudakov} asserts that $D_2(K_5)=1/4$, i.e., every $K_5$-free graph on $n$ vertices can be made bipartite by removing at most $n^2/8$ edges. The case  $r=6$ was verified by Hu, Lidick\'y, Martins, Norin, and Volec~\cite{flagmulti} using flag algebras. 
Variants of weighted Tur\'an-type problems were studied previously, see~\cite{balogh2025weighted} and \cite{balogh2024ramsey}.
We propose a weighted extension of this Tur\'an-type problem as follows.

\begin{question}\label{wk5}
Let $a\ge b\ge 2$ be integers. Let $R$ be a $K_a$-free weighted graph on $n$ vertices, where the weight of every edge in a $K_b$ is $1/2$, the weight of other edges is $1$, and the weight of non-edges is $0$.
What is the minimum weight of a cut in $R$, i.e., what is the minimum total weight of edges that are needed to be removed to make $R$ bipartite? 
\end{question}

The answer to~\Cref{wk5} 
is $n^2/18$ if $a=4$ and $b=3$.
It is stated as Theorem~\ref{wk5up}, which implies our main result,~\Cref{k5}.

\begin{theorem}\label{wk5up}
    Let $G$ be a $K_4$-free weighted graph on $n$ vertices, where the weight of every edge in a triangle is $1/2$, the weight of other edges is $1$, and the weight of non-edges is $0$.
   The graph $G$ can be made bipartite by removing edges of total weight at most $n^2/18$.   
\end{theorem}

Note that if $G$ is triangle-free, then every edge has weight $1$. Hence, in that case Question~\ref{wk5} reduces to the classical triangle-free setting of determining $D_2(K_3)$. For a long time, the best upper bound was $n^2/18$, which was improved by Balogh, Clemen, and Lidick\'y~\cite{balogh2021max} to $n^2/23.5$. Therefore, it suffices to handle the case when $G$ contains a triangle. If every edge is in a triangle, then every edge has weight $1/2$. The only condition we have now is that $G$ is $K_4$-free, hence in this case the result of Sudakov~\cite{MR2359832} solves the question. Therefore, Theorem~\ref{wk5up} relates the two problems.

\medskip

The \emph{Ramsey-Tur\'an number} $RT(n,H,m)$ is the maximum number of edges in an $n$-vertex $H$-free graph~$G$ with $\alpha(G) \leq m$.
This function was introduced by S\'os~\cite{SosRT}, and Erd\H os and S\'os~\cite{ErdosSosRT}, who studied the case when $H=K_q$ is a clique and $m=m(n)=o(n)$ is a sublinear function in $n$. Since then, determining the Ramsey-Tur\'an number for various $H$ and $m$ has become a classical topic in extremal combinatorics. In particular,  $RT(n,K_q,o(n))$ has been determined for every $q\geq 3$, see~\cite{BolobasErdos,MR1300968,ErdosSosRT,Szemeredi}.

In this paper, we initiate connecting these two types of problems, and raise the following general problem. Similar questions have been studied for other classical extremal problems; see, e.g.,~\cite{baloggeneralizedRT, balogh2018triangle, BMS2016,  Han2021ART,  kim2018conjectures, nenadov2020ramsey}.

\begin{problem}\label{meta}
Let $G$ be an $n$-vertex $H$-free graph with a sublinear independence number. At most how many edges are needed to be removed to make it bipartite?
\end{problem}

\begin{definition}
Given a graph $H$, let $D_2^\#(n, H, m)$ be the minimum number of edges that are needed to be removed from every $n$-vertex $H$-free graph $G$ with $\alpha(G)\le m$ to make $G$ bipartite. Define
$$d_2(H)\coloneqq \lim_{\delta\to 0}\limsup_{n\to \infty} \frac{D_2^\#(n,H,\delta n)}{\binom{n}{2}}.$$
\end{definition}

It is easy to see that $d_2(H)$ is well-defined, and we expect that the corresponding limit exists even without taking the limit superior. Throughout the paper, we ignore the $o(1)$ terms, whenever this causes no ambiguity.
We are particularly interested in the case when $H$ is a clique.
For $H=K_3$, since every $n$-vertex triangle-free graph with sublinear independence number has $o(n^2)$ edges, we immediately get that $d_2(K_3)=0$. When $H$ is a larger clique, we have the following proposition, which will be proved in Section~\ref{sec:prelim}.

\begin{proposition}\label{prop: K_4}
For every $q\geq 2$, we have\\
\noindent(i) $$d_2(K_{q+1})\geq\frac{D_2(K_{q})}{2};$$
\noindent(ii) $$d_2(K_{q+1})\leq D_{2}(K_{q});$$
\noindent(iii) $$d_2(K_{4})\ = \ \frac{D_2(K_{3})}{2}.$$
\end{proposition}

Proposition~\ref{prop: K_4} (iii) claims that determining $d_2(K_{4})$ is equivalent to determining $D_2(K_{3})$.
Our main result settles Problem~\ref{meta} when $H=K_5$.

\begin{theorem}\label{k5}
$$d_2(K_5)=\frac{1}{9}.$$
\end{theorem}

We prove a stability version of Theorem~\ref{k5}, stated as Theorem~\ref{lemma:fah}. We place it in the Appendix A, as its proof relies entirely on the method of flag algebra.
Part of our proof of Theorem~\ref{k5} also uses flag algebras, through only three simple cuts. 
%Hence, it would be interesting to find an entirely combinatorial proof of Theorem~\ref{k5}.

Four months after we submitted our paper, ChatGPT found an entirely combinatorial proof of Theorem~\ref{k5}, which we include in the Appendix D.

We emphasize that Problem~\ref{meta} is open even for cliques of size at least $8$.

\begin{problem}
Determine $d_2(K_q)$ for $q\geq 8$.
\end{problem}

When we originally submitted our paper, we were aware of the constructions listed in
Table~\ref{tab:pattern}, but did not have any evidence that they would be best possible.  These constructions are analogues of results in~\cite{MR1300968,ErdosSosRT}.

Now, we utilized Claude to inspect constructions for $d_2(K_q)$ to see if we could find some better ones. 
The general results are summarized in Table~\ref{tab:pattern}, and the
 exceptions for small $q$ are described in Table~\ref{tab:exceptions}. 

We searched graphs formed  by heavy joins of some of  $K_1$, light $K_2$, light $C_5$, light $K_3$, light $K_4$ and the octahedron (defined later), computed for $4\le q\le 300$.
The actual constructions are obtained by blow-ups, where heavy edges are replaced by complete bipartite graphs, and light edges are replaced by Bollob\'as-Erd\H os graphs, defined later in Section~\ref{sec:prelim}.
Finally, inside each of the blow-ups are triangle-free graphs with independence number $o(n)$.

\begin{table}[ht]
\centering
\renewcommand{\arraystretch}{2.2}
\begin{tabular}{c|c|l|c}
$q \bmod 4$ & $q$ & construction (heavy join) & value \\
\hline
$0$ & $4k$   & $K_{2k-2}$ + light $K_2$, \; $h:c = 1 : 2-\sqrt2$ & $\dfrac{(k-1)(k+2-2\sqrt2)}{2(k+1-\sqrt2)^2}$ \\
$1$ & $4k+1$ & $K_{2k}$ & $\dfrac{k-1}{2k}$ \\
$2$ & $4k+2$ & $K_{2k-1}$ + light $K_2$, \; equal weights & $\dfrac{2k^2-1}{(2k+1)^2}$ \\
$3$ & $4k+3$ & $K_{2k+1}$ & $\dfrac{2k^2}{(2k+1)^2}$ \\
\end{tabular}
\caption{General constructions for $d_2(K_q)$. Here $K_m$ denotes heavy join of $m$ classes (a complete $m$-partite graph whose classes are triangle-free with sublinear independence number), light $K_2$ denotes a Bollob\'as--Erd\H{o}s pair, and $h$, $c$ are the weights of the $K_1$ classes and of the classes of the light $K_2$.
}
\label{tab:pattern}
\end{table}

% \begin{table}[ht]
% \centering
% \renewcommand{\arraystretch}{1.8}
% \begin{tabular}{c|l|c|c}
% $q$ & best construction & value & value in Table~\ref{tab:pattern} \\
% \hline
% $4$  & light $C_5$ & $\dfrac{1}{25}$ & $0$ \\
% $5$  & light $K_3$ & $\dfrac{1}{9}$ & $0$ \\
% $6$  & octahedron & $\dfrac{1}{6}$ & $\dfrac{1}{9}$ \\
% $8$  & $K_2$ + light $C_5$ & $\frac{1342+65\sqrt{17}}{6889}\approx 0.23371$ & $\frac{10+\sqrt2}{49}\approx 0.23294$ \\
% $9$  & $K_1$ + $2$ light $K_2$ & $\frac{32}{121}\approx 0.26446$ & $\dfrac14$ \\
% $13$ & $K_3$ + $2$ light $K_2$ & $\frac{17890-180\sqrt5}{52441}\approx 0.33347$ & $\dfrac13$ \\
% \end{tabular}
% \caption{Values of $q\ge 4$ where the best construction differs from Table~\ref{tab:pattern}.}
% \label{tab:exceptions}
% \end{table}

\begin{table}[ht]
\centering\small
\renewcommand{\arraystretch}{1.6}
\begin{tabular}{c|l|c|c}
$q$ & best construction & value & value in Table~\ref{tab:pattern} \\
\hline
$4$  & light $C_5$ & $\dfrac{1}{25}$ & $0$ \\
$5$  & light $K_3$ & $\frac{1}{9}$ & $0$ \\
$6$  & octahedron $=F(3,2)$ & $\dfrac{1}{6}$ & $\dfrac{1}{9}$ \\
$8$  & $K_2$ + light $C_5$ & $\frac{1342+65\sqrt{17}}{6889}\approx 0.23371$ & $\frac{10+\sqrt2}{49}\approx 0.23294$ \\
$9$  & $F(5,3)$, equal weights & $\dfrac{4}{15}\approx 0.26667$ & $\dfrac14$ \\
$10$ & $G_{10}$ (see text) & $\frac{109111+1608\sqrt3}{391876}\approx 0.28554$ & $\frac{7}{25}$ \\
$13$ & $K_3$ + $2$ light $K_2$ & $\frac{17890-180\sqrt5}{52441}\approx 0.33347$ & $\dfrac13$ \\
\end{tabular}
\caption{Values of $q\ge 4$ where the best construction found differs from Table~\ref{tab:pattern}.
In the weighted reduced-graph language, $F(a,b)$ is the graph on $[b]\times[a]$ in which
$(c,i)(c,j)$ is a light edge for $i\ne j$, $(c,i)(d,j)$ is a heavy edge for $c\ne d$ and
$i\ne j$, and $(c,i)(d,i)$ is a non-edge for $c\ne d$.
The graph $G_{10}$ is a graph on 11 vertices indexed by four index sets $\{1,2\}$,
$\{1,2,3\}$, $\{1,2,4\}$, $\{1,2,5\}$ (light edges inside each set, complete between sets with  different
indices, empty between equal indices) with weight $\approx0.1105$ on the two classes of the first copy,
$\approx0.0386$ on the other classes with indices $1,2$, and $\approx0.1825$ on the classes with indices $3,4,5$.}
\label{tab:exceptions}
\end{table}

%\textbf{The following is the old paragraph, now obsolete.}
%When $q=4k+1\geq 7$, the following construction is a natural candidate for the lower bound of $d_2(K_q)$:
%take a balanced complete $2k$-partite graph on $n$ vertices and embed into each class a triangle-free graph with independence number $o(n)$.
%For other values of $q$, one can obtain a family of constructions in the following way: let $V_1,\ldots,V_j$ be the classes of some vertex partition. Between some pair of classes, we have a complete bipartite graph; between the remaining pairs, we have a Bollob\'as-Erd\H os graph, defined later in Section~\ref{sec:prelim}.
%Inside each $V_i$, we have a class of the Bollob\'as-Erd\H os graph, a triangle-free graph with independence number $o(n)$.
%For each such construction, its clique-number could be easily computed. To compute how many edges are needed to be  removed to make it bipartite, we would need a careful optimization, as now the classes are not necessarily of the same size. As we do not have good upper bounds, we did not attempt to optimize the construction.
%For $q=6$, we think that the partition $V_1,V_2,V_3$ where $(V_1,V_2)$ forms a Bollob\'as-Erd\H os graph and $(V_1,V_3)$, $(V_2,V_3)$ are complete bipartite graphs with $|V_1|=|V_2|=2n/5$, $|V_3|=n/5$ is optimal, and $d_2(K_6)=4/25$.

The case $q=4$ is practically Proposition~\ref{prop: K_4} (iii), and $q=5$ is Theorem~\ref{k5}.
The values for $q=6,8,9,10,13$ were given by ChatGPT.
For $q=6$, ChatGPT found the following octahedron construction. Take two disjoint copies of the $3$-Bollob\'as-Erd\H{o}s graph defined in Section~\ref{sec:prelim}, with classes $A_1,A_2,A_3$ and $B_1,B_2,B_3$, respectively, each of size $n/6$.
Between the two copies, add all edges between $A_i$ and $B_j$ for $i\ne j$.
Counting triangles that meet both copies and all three sets $A_i\cup B_i$ gives $d_2(K_6)\geq 1/6$.
%That inspired us to ask Claude to explore larger values of $q$ and the results are in Tables~\ref{tab:pattern} and \ref{tab:exceptions}.

After finding the construction for $q=6$, we asked Claude to construct flag algebras calculations giving matching upper bounds for $q\in\{6,7\}$ and obtained the following theorem.

\begin{theorem}\label{67}
$d_2(K_6)=1/6$ and $d_2(K_7)=2/9$.
\end{theorem}

The flag algebra calculations were checked by a program written by the last author and by another independent verification written by ChatGPT.
The calculations for $d_2(K_6)$ and $d_2(K_7)$ are given in the Appendices B and C, respectively.

We suspect that all extremal constructions for $d_2(K_q)$ are built from small templates.
Let $W$ be a graph with edge weights in $\{0,\frac12,1\}$.
A blow-up of $W$ is an $n$-vertex graph where each vertex of $W$ is replaced by a triangle-free graph with independence number $o(n)$. 
Edges of $W$ of weight $\frac12$ are replaced by copies of Bollob\'as-Erd\H{o}s graphs and edges of weight $1$ are replaced by complete bipartite graphs.
Notice the blow-up does not have to be balanced.

\begin{conjecture}\label{conj:121}
For each $q$ and $n$ there exists $W$ with edges weighted by $\{0,\frac12,1\}$ such that blow-ups $B_n$ of $W$ give tight constructions for $d_2(K_q)$
as $n \to \infty$.
\end{conjecture}

The data in Table~\ref{tab:pattern} show that the blow-up could be unbalanced. 
This adds more computational complexity to the problem.

%The resulting graph has sublinear independence number, since each class does.
%Each $A_i\cup B_i$ is triangle-free, so a $K_6$ would have to meet each $A_i\cup B_i$ in an edge. Sicne there are no edges between each $A_i\cup B_i$, two of these edges would lie in the same copy, yielding a $K_4$ in the union of two classes of that copy, a contradiction.
%Removing all edges within the two copies makes the graph bipartite and removes $(1/12+o(1))n^2$ edges. To see that one cannot do much better, count the triangles that meet both copies and all three sets $A_i\cup B_i$. Each edge between distinct classes of one copy lies in exactly $n/6$ such triangles, so there are $(1/72+o(1))n^3$ of them. By the degree property of the Bollob\'as-Erd\H{o}s graph, each edge between the two copies lies in at most $(1+o(1))n/6$ such triangles. Thus at least $(1/12-o(1))n^2$ edges must be removed to eliminate all these triangles, giving $d_2(K_6)\geq 1/6$.

The rest of the paper is organized as follows. In Section~\ref{sec:prelim} we show that to prove Theorem~\ref{k5}, it is sufficient to solve  a weighted Tur\'an-type problem, i.e., sufficient to answer~\Cref{wk5} when $a=4$ and $b=3$; in Section~\ref{sec:human} we prove the corresponding weighted Tur\'an-type problem when the graph has large edge density; in Section~\ref{sec:FA} we handle the case when the graph has small edge density, where we use flag algebras.

\section{Preliminary}\label{sec:prelim}

We will need the following classical construction.
Denote by $G\coloneqq \mathrm{BE}(n,s,x)$ the $s$-Bollob\'as--Erd\H os graph introduced in~\cite{balogh2017sBE}, defined as follows. The case when $s=2$ is the well-known Bollob\'as--Erd\H os graph~\cite{BolobasErdos}. For further details on the Bollob\'as--Erd\H os graph, see also~\cite{maya2025generalized, liu2026geometric}.

Fix $\varepsilon>0$ and let $d$ be sufficiently large. Let $P$ be an $n/s$-element point set on a $d$-dimensional unit sphere with radius $1$. Put $x=\varepsilon/\sqrt{d}$.
Take $s$ disjoint copies $X_1,\dots,X_s$ of $P$ as the vertex classes of $G$.
For $u\in X_i$ and $v\in X_j$, put an edge $uv\in E(G)$ if $d(u,v)>2-x$ when $i=j$, and if $d(u,v)<\sqrt2-x$ when $i\neq j$.
Each class $X_i$ spans a triangle-free graph with independence number $o(n)$ and with $o(n^2)$ edges.
For every $i\neq j$, the pair $X_i\cup X_j$ induces a copy of the Bollob\'as--Erd\H os graph, which is $K_4$-free as proved in~\cite{BolobasErdos}. Therefore, $G$ is $K_{s+2}$-free.
Moreover, every induced subgraph $G[X_i,X_j]$ has edge density $1/2+o(1)$, as each vertex $x\in X_i$ has $(1/2-o(1))n/s$ neighbors in $X_j$. Furthermore, for every $\beta>0$, if $n$ and $d$ are sufficiently large, then 
$G[X_i,X_j]$ is a $\beta$-regular bipartite graph with density $1/2+o(1)$.
(Recall that it means that every $X'_i\subset X_i, \ X'_j\subset X_j$ with $|X'_i|>\beta |X_i|,\ |X'_j|>\beta |X_j|$ satisfy that the density of $G[X'_i,X'_j]$ is between $1/2-\beta$ and $1/2+\beta$.)

\begin{proof}[Proof of Proposition~\ref{prop: K_4}]

(i) Our goal is to construct an $n$-vertex $K_{q+1}$-free graph $G$ with $\alpha(G)=o(n)$ such that at least $\left(D_2(K_{q})/2+o(1)\right)\binom{n}{2}$ edges must be removed to make $G$ bipartite.
Choose $m$ and $n$ sufficiently large so that
$k:=n/m$ is also sufficiently large. To simplify calculation, we shall assume that $k$ is an integer.
 Let $H$ be a $K_{q}$-free graph on $m$ vertices such that every bipartite subgraph of $H$ misses at least $\left(D_2(K_{q})+o(1)\right)\binom{m}{2}$ edges of $H$.
We now construct $G$ from $H$ as follows.
Replace each vertex $x\in V(H)$ with a set $V_x$ of $k$ vertices and call it a cluster. Then, $G$ has $n=mk$ vertices.
We view the family of clusters as being equipped with a global Bollob\'as--Erd\H os-type structure, so that for every edge $xy\in E(H)$, the corresponding clusters $V_x\cup V_y$ induce a Bollob\'as--Erd\H os graph, and for every copy of $K_{q-1}$ in $H$, the union of the corresponding $q-1$ clusters induces a $(q-1)$-Bollob\'as--Erd\H os graph.

Note that every $G[V_x, V_y]$ is $K_4$-free and $\alpha(G)=o(n)$ by the definition of $G$. Observe that every clique $S$ in $G$ naturally corresponds to a clique in $H$. Furthermore, as the clusters are triangle-free, each cluster might contain at most two vertices of $S$. The property that every $G[V_x, V_y]$ is $K_4$-free implies that there could be at most one cluster containing two vertices of $S$, implying that  $G$ is $K_{q+1}$-free as $H$ is $K_{q}$-free. Hence, it suffices to show that at least $\left(D_2(K_{q})/2+o(1)\right)\binom{n}{2}$ edges must be removed to make $G$ bipartite.
Note that the total number of edges inside clusters is $m\cdot o(k^2)=o(n^2)$ when $k\to \infty$, which is subquadratic. Therefore, we may delete all such edges and work with the subgraph $G'$ obtained from $G$ by keeping only the edges between distinct clusters. This affects the  number of edges that must be removed to make $G$ bipartite by at most $o(n^2)$, thus it is sufficient to prove the desired lower bound $\left(D_2(K_{q})/2+o(1)\right)\binom{n}{2}$ for $G'$.

Let $(A,B)$ be a bipartition of $V(G')$ maximizing the number of edges between $A$ and $B$, and among all such bipartitions choose one for which the number of clusters intersecting both $A$ and $B$ is minimum.

If every cluster is entirely contained in $A$ or in $B$, then $(A,B)$ induces a bipartition of $H$ naturally. Every bipartite subgraph of $H$ misses at least $D_2(K_{q})\binom{m}{2}$ edges by the choice of $H$.
Each such edge $xy\in E(H)$ corresponds to a pair $(V_x, V_y)$ with edge density $1/2+o(1)$ in $G'$. Therefore, to make $G'$ bipartite, the number of edges that must be removed is at least
\[
\left(\frac12+o(1)\right)\cdot k^2\cdot \left(D_2(K_{q})+o(1)\right)\cdot\binom{m}{2}
=
\left(\frac{D_2(K_{q})}{2}+o(1)\right)\binom{n}{2}.
\]

Now we may assume that some cluster $C$ is split by $(A,B)$.
For a cluster $V_x$ define $V_x^A:= A\cap V_x$ and $V_x^B:= B\cap V_x$.
Split $V_x^A$ and $V_x^B$ into classes of sizes $\beta k$. The total number of leftover vertices is at most $\beta n$.
At the end we might just assign them randomly, and they do not impact the size of the cut too much, only by $o(n^2)$. 
By the $\beta$-regularity, the density of pairs of those new subclasses are $1/2+o(1)$.
Now, if for one cluster, two of its subclasses are not in the same class of the bipartition $(A,B)$, then we can move them all either into $A$ or $B$, without decreasing the number of cross-edges in the $[A,B]$-cut with more than $o(k^2m)$. Doing this for every $V_x$, we obtain a bipartition, where each cluster is inside one of the classes, and the number of cross-edges were decreased by at most $o(n^2)$.

(ii) When $q=2$, the inequality is trivial, as a triangle-free graph with sublinear independence number has subquadratically many edges. For $q\geq 3$, let $G$ be an arbitrary $K_{q+1}$-free graph on $n$ vertices with sublinear independence number. Apply Szemer\'edi's Regularity Lemma to $G$ to obtain a cluster graph $R$ on $m=O(1)$ vertices. It is standard to show that $R$ is $K_{q}$-free and every edge in a copy of $K_{q-1}$ in $R$ has edge density (in $G$) at most $1/2+o(1)$. (See Theorem~\ref{cluster} below for the case $q=4$.)
Hence, we can remove at most $D_2(K_{q})\cdot \binom{m}{2}$ edges from $R$ to make it bipartite, so the number of edges we need to remove from $G$ to make it bipartite is at most 
\begin{equation}\label{eqn: d-D}
    D_2(K_{q})\cdot \binom{m}{2}\cdot \left(\frac{n}{m}\right)^2+o(n^2)=D_2(K_{q})\cdot \binom{n}{2}+o(n^2),
\end{equation}
implying that
\[
d_2(K_{q+1})\leq D_{2}(K_{q}).
\]

(iii) By (i), we have $d_2(K_4)\geq D_2(K_3)/2$. For the other direction, assume that  $G$ is a $K_4$-free graph, with sublinear independence number.
Szemer\'edi~\cite{Szemeredi}  proved that the
 cluster graph $R$ is triangle-free and every edge  of $R$ has density in $G$ at most $1/2+o(1)$.
Now we can save a $1/2$ factor in~\eqref{eqn: d-D}, implying that
\[
d_2(K_4)\leq \frac{D_2(K_3)}{2}. \qedhere
\]
\end{proof}

\subsection{Proof of Theorem~\ref{k5}}

The graph $\mathrm{BE}(n,3,x)$ is $K_5$-free with sublinear independence number.
In $\mathrm{BE}(n,3,x)$, removing all edges inside the classes $X_i$ and between $X_1$ and $X_2$ provides a bipartite graph. The number of removed edges is $(1/18+o(1))n^2$. We show that one cannot do much better, which gives that $d_2(K_5)\geq 1/9$. 
Indeed, the properties of the Bollob\'as-Erd\H os graph imply that the number of triangles in $\mathrm{BE}(n,3,x)$ is at least $(1/8+o(1))n^3/27$.
Moreover, all but $o(n^2)$ edges are contained in at most $(1/12+o(1))n$ triangles, while every edge is contained in at most $n$ triangles.
Thus, the exceptional $o(n^2)$ edges are contained in only $o(n^3)$ triangles. 
Therefore, at least \[ \frac{\left(\frac18+o(1)\right)\frac{n^3}{27}} {\left(\frac{1}{12}+o(1)\right)n} = \left(\frac{1}{18}+o(1)\right)n^2 = \left(\frac{1}{9}+o(1)\right)\binom{n}{2}\] edges must be removed to eliminate all triangles.

For the rest of the paper, we will focus on the proof of $d_2(K_5)\leq 1/9$.
Let $G$ be a $K_5$-free graph on $n$ vertices with $\alpha(G)=o(n)$. Apply Szemer\'edi's Regularity Lemma to $G$ to obtain a cluster graph $R$ on $m$ vertices. Extending the method of Szemer\'edi~\cite{Szemeredi}, the following result appeared in many papers.
    
\begin{theorem}\label{cluster}
(i) The cluster graph $R$ contains no $K_4$.\\
(ii) Every edge  of $R$ which is in a triangle of $R$ has density in $G$ at most $1/2+o(1)$.
\end{theorem}

Hence, we transfer our problem to Problem~\ref{wk5} and it suffices to prove Theorem~\ref{wk5up}.

\subsection{Definitions and Notation}\label{sec: def}

An edge is {\it heavy} if it has weight $1$, and {\it light} if it has weight $1/2$.
For a vertex $v$,  denote by $a_v$ the number of heavy edges incident to $v$, and $A_v$ the set of vertices connected by a heavy edge to $v$. Similarly, denote by $b_v$ the number of light edges incident to $v$, and $B_v$ the set of vertices connected by a light edge to $v$.
Denote by $d_v$ the degree of $v$ and $d^w_v$ the weighted degree of $v$, then
\[d_v=a_v+b_v \text{\quad\quad and\quad\quad} d^w_v=a_v+\frac{b_v}{2}.\]
Write $e_h$ for the number of the heavy edges in a graph $G$, and $e_h^w$ for the total weight of the heavy edges of $G$. Similarly, write $e_\ell$ for the number of the light edges in $G$, and $e_\ell^w$ for the total weight of the light edges of $G$. Denote by $e^w$ the total weight of all edges of $G$, then
\[e(G)=e_h+e_\ell, \ e^w=e_h^w+e_\ell^w, \ e_h=e_h^w, \ e_\ell=2e_\ell^w, \ \sum_v a_v=2e_h=2e_h^w, \\ \sum_v b_v=2e_\ell=4e_\ell^w.\]
Denote $e^w_v$ the total weight of the edges in $N(v)$.
As each such edge has weight $1/2$, it is half of the number of edges in $N(v)$.
Denote by $d_{u,v}$ the number of common neighbors of vertices $u$ and $v$.
Denote by $B(G)$ the total weight of the edges of the largest bipartite subgraph of $G$. Let $n$ be the number of vertices and $m$ be the number of triangles in $G$.
Recall that we want to show $e^w-B(G)\leq n^2/18$.

We say that a graph $G$ is {\it bad} if $G$ is an $n$-vertex $K_4$-free graph, whose triangles contain no heavy edges and every cut in $G$ has weight at least $n^2/18$.

\section{
\texorpdfstring{The case when $e^w\ge n^2/6$}{The case when ew >= n²/6}
}\label{sec:human}

In this section we eliminate  the case when the  total edge weight is large.
\begin{lemma}\label{lemma:heavy}
    If $G$ is a bad graph on $n$ vertices, then $e^w \leq n^2/6$.
\end{lemma}

In the proof we use only one cut,  the bipartition of $N(v)$ and $V(G)\setminus N(v)$ for every vertex $v$.
However, the simple proof that it works is (slightly) different from the methods of Sudakov~\cite{MR2359832}, and Erd\H{o}s, Faudree, Pach and Spencer~\cite{erdos1988makebi}.

\begin{proof}[Proof of Lemma~\ref{lemma:heavy}]

Let $G$ be a bad graph on $n$ vertices, with arbitrary edge weight.

Let $T_1$ be the number of ordered pairs $(uv,w)$ such that $uv$ is a heavy edge and neither $uw$ nor $vw$ is an edge. Since heavy edges are not contained in triangles,
\[
T_1=\sum_{uv\text{ heavy}}(n-d_u-d_v).
\]
Note that each vertex $v$ is incident with $a_v$ heavy edges, yielding
\[
T_1=e_h(n-2)-\sum_v a_v (d_v-1)=e_h(n-2)-\sum_v a_v d_v + 2e_h=e_hn-\sum_v a_v d_v.
\]

Let $T_2$ be the number of ordered pairs $(uv,w)$ such that $uv$ is a light edge and neither $uw$ nor $vw$ is an edge, then
\[
T_2=\sum_{uv\text{ light}}(n-d_u-d_v+d_{u,v}).
\]
Note that each $v$ is incident with $b_v$ light edges. 
Recalling that $m$ is the number of triangles in $G$, we have
\[
T_2=e_\ell(n-2)-\sum_v b_v (d_v-1)+3m=e_\ell(n-2)-\sum_v b_v d_v+2e_\ell+3m=e_\ell n-\sum_v b_v d_v+3m.
\]

For a vertex $w\in V(G)$, the number of edges inside $N(w)$ equals to the number of triangles containing $w$, while the number of edges inside $V(G)\setminus N(w)$ is counted by the triples $T_1$ and $T_2$.
As $G$ is a bad graph, for every $w$, the cut $(N(w),V(G)\setminus N(w))$ contains at least $n^2/18$ edges. 
Summing over every $w$ yields
\begin{equation}\label{eq: sum}
\frac{n^3}{18}
\le
\frac{3}{2}m
+ \underbrace{e_hn-\sum_v a_v d_v}_{T_1}
+ \frac{1}{2}\underbrace{\left(e_\ell n-\sum_v b_v d_v+3m\right)}_{T_2}.
\end{equation} The term $3m/2$ is there, because every triangle is counted three times but every edge of it is light, therefore has weight $1/2$.

Rewriting~\eqref{eq: sum} with  
$a_v=d_v^w-b_v/2$, 
$d_v=d_v^w+b_v/2$
and $e^w=e_h+e_\ell/2$ gives
\begin{equation}\label{eqn: simplesum}
    \frac{n^3}{18}
\le
3m + e^w\cdot  n - \sum_v (d_v^w)^2 - \frac{1}{2}\sum_v b_v d_v^w.
\end{equation}
Using the Cauchy--Schwarz Inequality and $\sum_v d_v^w=2e^w$, we have
\[
\sum_v (d_v^w)^2 \ge \frac{4(e^w)^2}{n}.
\]
Noticing that $d_v^w= a_v+b_v/2\geq b_v/2$, we have
\begin{equation}
\frac{n^3}{18}
\le
3m + e^w  n - \frac{4(e^w)^2}{n} - \frac{1}{4}\sum_v b_v^2
-\frac{1}{2}\sum_v a_v b_v
\le
3m + e^w  n - \frac{4(e^w)^2}{n} - \frac{1}{4}\sum_v b_v^2.
\end{equation}
Since every triangle contributes to exactly three terms in $\sum_v b_v^2$, and the graph induced by $B_v$ is triangle-free and thus contains at most $b_v^2/4$ edges, we have
\[
3m \le \sum_v \frac{b_v^2}{4},
\]
yielding
\[
\frac{n^3}{18}
\le
e^w n - \frac{4(e^w)^2}{n}.
\]
Let $x=e^w/n^2$. Dividing by $n^3$ gives
\[
\frac{1}{18}\le x-4x^2,
\]
which is equivalent to
\[
(12x-1)(6x-1)\le 0.
\]
Thus $x\le 1/6$, implying $e^w \leq n^2/6$. 
\end{proof}

\section{Flag algebra proof of~\Cref{wk5up}}\label{sec:FA}

In this section we prove Lemma~\ref{lemma:fa} using flag algebras, a strengthening of the lemma will be proved in the Appendix.

\begin{lemma}\label{lemma:fa}
    For every convergent sequence  $(G_n)_n$ of bad graphs, 
    the density of heavy edges in $G_n$ converges to zero.
\end{lemma}

A standard blow-up argument, similar to the proof of~\Cref{prop: K_4}, transfers the result to a finite fixed graph.

\begin{corollary}\label{lemma:G}
If $G$ is a bad graph, then it has no heavy edges.
\end{corollary}

\begin{proof}[Proof of Corollary~\ref{lemma:G} assuming Lemma~\ref{lemma:fa}]
Let $G$ be an $n$-vertex bad graph.
Let $G_k$ be obtained from $G$ by replacing each vertex of $G$ by an independent set of $k$ vertices, then $G_k$ has $nk$ vertices.
Each cut that was in $G$ yields a cut in $G_k$ of weight $(nk)^2/18$.
The sequence $(G_k)_k$ is a convergent sequence.

In order to apply Lemma~\ref{lemma:fa} we need to verify that $G_k$ does not have a cut of weight less than $(nk)^2/18$.
Call $k$ vertices in $G_k$ that were obtained from one vertex in $G$ a cluster.
Let $(A,B)$ be a bipartition of vertices of $G_k$ of the smallest weight such that the number of clusters intersecting both $A$ and $B$ is minimized.
If each cluster is entirely contained in $A$ or $B$, the bipartition corresponding to a bipartition in $G$ and $G$ being bad implies that the weight is at least $(nk)^2/18$.
Let $C$ be a cluster in $G_k$ intersecting both $A$ and $B$.
Our goal is to show that at least one of the bipartitions  $(A \cup C, B\setminus C)$ and $(A \setminus C, B\cup C)$ contradicts the choice of $(A,B)$.

For a vertex $v$, denote by $w(v)$ the sum of weights of edges connecting $v$ to the  other vertices in the same part as $v$ of the bipartition. 
Notice that $w(v)$ could be interpreted as the contribution of $v$ to the weight of the cut.
Let $v_a$ be a vertex in $A \cap C$ and $v_b$ be a vertex in $B \cap C$.
If $w(v_a) > w(v_b)$, then moving $v_a$ from $A$ to $B$ decreases the overall weight of the cut, which contradicts the minimality of the cut.
By the symmetry between $A$ and $B$, we conclude $w(v_a) = w(v_b)$. In particular,
$w(v)$ is the same for every $v \in C$.
Since $C$ is an independent set, the bipartition $(A \cup C, B\setminus C)$ has the same weight as $(A,B)$ while having fewer clusters intersecting  both parts, which contradicts the choice of $(A,B)$.
\end{proof}

\begin{proof}[Proof of Lemma~\ref{lemma:fa}]
We model the problem in flag algebras using 2-edge-colored graphs.
Light edges are depicted as \Fe202 and heavy edges are depicted as \Fe203. 
No $K_4$'s and no triangles with heavy edges are modeled by forbidding the following induced flags:  

\begin{align}\label{eq:forb}
\Fe30{2 2 3},
\quad 
\Fe30{2 3 3},
\quad
\Fe30{3 3 3},
\quad
\Fe40{2 2 2 2 2 2}.
\end{align}
Since every cut in $G_n$ has weight at least $n^2/18$, the following particular three types of cuts have weight $n^2/18$ and give inequalities that can be used in flag algebras calculations.  
Since densities of edges are scaled  by $\binom{n}{2}$ rather than $n^2$, we use $1/9$ for the weight of a cut which corresponds to $n^2/18$.

(1) Partition the vertices randomly into two parts. 
    \begin{align}\label{eq:cut1}
\frac19 \leq \frac12\left( \frac12 \Fuu2 + \Fuu3 \right)
    \end{align}

    (2) Fix a vertex $v$ and partition the remaining vertices into two classes: neighbors of $v$ and non-neighbors of $v$.
     \begin{align}\label{eq:cut2}
 \frac19  \leq   \Fluu113 + \frac12 \Fluu112 + \frac12 \Fluu222 
    \end{align}

    (3)  Fix a heavy edge $uv$.
    Notice $N(u) \cap N(v) = \emptyset$.
    Let $X := V(G) \setminus (N(u) \cup N(v))$.
    Let $Y$ be a subset of $X$ of size $\lfloor|X|/2\rfloor$ sampled uniformly at random.
    Partition the vertices of $G_n$ into 
    $A = N(u) \cup Y$, $B = N(v)\cup  (X \setminus Y)$. 
    The expectation over $Y$ gives the following inequality in flag algebras.
    
    \begin{align}\label{eq:cut3}
 \frac19  \leq  &~ \frac12 \Flluu322112 + \frac12 \Flluu311222 + \frac12 \Flluu331113 + \frac14 \Flluu331112 + \frac12 \Flluu321113 + \frac14 \Flluu321112 \\ \nonumber
 &+ \frac12 \Flluu311313 + \frac14 \Flluu311312 + \frac12 \Flluu311213 + \frac14 \Flluu311212 + \frac12 \Flluu311113 + \frac14 \Flluu311112 
    \end{align}

Flag algebras calculation on $5$ vertices
combining 
\eqref{eq:forb},
\eqref{eq:cut1},
\eqref{eq:cut2}, 
\eqref{eq:cut3},
and squares of flags of five types
\Fe11{}, \Flll111, \Flll112, \Flll113 and \Flll122
implies
\[
 \Fuu3 = 0.
\]
Since the particular coefficients the squares and not particularly enlightening, the coefficients for the calculation are available as supplemental files in machine friendly format as well as in \LaTeX{} at \oururl.  
Note that in the non-weighted setting~\cite{erdos1988makebi, MR2359832}, similar cuts were used.
\end{proof}

\begin{proof}[Proof of~Theorem~\ref{wk5up}.]
Let $G$ be a $K_4$-free $n$-vertex weighted graph, where the weight of every edge in a triangle is $1/2$, the weight of the  other edges is $1$, and the weight of non-edges is $0$.
If $G$ has a cut of weight at most $n^2/18$, then we are done, hence we  may
assume that every cut in $G$ has weight at least $n^2/18$. By Corollary~\ref{lemma:G}, $G$ has no edges of weight $1$, 
hence every edge of $G$ has weight $1/2$.

Since $G$ is $K_4$-free, Theorem~\ref{thm:sudakov}
implies that $G$ can be made bipartite by removing at most $n^2/9$ edges.
Since each edge has weight $1/2$, this gives a cut of weight at most $n^2/18$.
\end{proof}

\section*{AI Declaration}
We thank Bowen Li, for running the AI programs, which found an improved construction for $K_6$ and a non-flag proof of Theorem~\ref{k5}. Afterwards,  the authors used AI to create the results in Table~\ref{tab:exceptions}. While, we knew the constructions of Table~\ref{tab:pattern} for a while, previously we have not run computer assisted search, to check if one could find better constructions. We used AI to prove Theorem~\ref{67}, and to replace the flag algebra part of the proof of~\Cref{wk5up} with a combinatorial proof. We put all those proofs  in Appendices B, C and D.
For the other parts of the paper, the authors have not used AI.

\section*{Acknowledgment}
The authors would like to thank Ramón Iván García Alvarez, Florian Pfender and Jan Volec for discussions during the early stages of this project.
The first and second author started discussing this problem at AIM SQuaRE Workshops ``Crossing numbers of complete and complete bipartite graphs''. 
This work used the computing resources at the Center for Computational Mathematics, University of Colorado Denver, including the Alderaan cluster, supported by the National Science Foundation award OAC-2019089.

\bibliographystyle{plainurl}
\bibliography{references}

\section*{Appendix A}

The purpose of this Appendix is to describe a strengthening of Lemma~\ref{lemma:fa} by proving a stability version. 
The complete balanced tripartite graph with every edge having weight $1/2$ requires a cut of weight $n^2/18$.
As this graph has no heavy edges, a natural stability result would be a better upper bound, which depends on the number of heavy edges.
More precisely, assume that the weighted graph has $h$ heavy edges. We aim to find a function $f:\mathbb{N} \to \mathbb{R}^+$ such that the upper bound  $n^2/18 - f(h)$ holds.

\begin{theorem}\label{lemma:fah}
    Let $(G_n)_n$ be a convergent sequence of graphs with light and heavy edges.
    If for all $n$  we have that $G_n$ is $K_4$-free, triangles in $G_n$ contain no heavy edges and every cut in $G_n$ has weight at least $n^2/18 - 0.008h$, where $h$ is the number of heavy edges in $G_n$, then 
    the density of heavy edges in $G_n$ converges to zero.
\end{theorem}
\begin{proof}
We use the proof of Lemma~\ref{lemma:fa}, while modifying \eqref{eq:cut1}, \eqref{eq:cut2} and \eqref{eq:cut3} by adding 0.016 of the density of heavy edges to the right-hand side.
The entire program is the following:
\[
(Ph)
\begin{cases}
    \text{maximize} & \Fe20{3} \\
    \text{subject to} &
  1  / 9  \leq  \frac{1}{2}\times  \Bigg(  \frac{1}{2} \Fe20{2} + \Fe20{3} \Bigg) + 0.016 \Fe20{3}   \\
  & 
  1  / 9  \leq  \Fe31{1 1 3} + \frac12 \Fe31{1 1 2} + \frac12 \Fe31{2 2 2} + 0.016\times \pi^\sigma\Bigg(\Fe20{3} \Bigg) 
 \\
  &  1  / 9  \leq  \frac12 \Fe42{3 2 2 1 1 2} + \frac12 \Fe42{3 1 1 2 2 2} + \frac12 \Fe42{3 3 1 1 1 3} + \frac14 \Fe42{3 3 1 1 1 2} + \frac12 \Fe42{3 2 1 1 1 3} + \frac{1}{4} \Fe42{3 2 1 1 1 2} + \frac12 \Fe42{3 1 1 3 1 3}  + \\ & + \frac{1}{4} \Fe42{3 1 1 3 1 2} + 
   \frac12 \Fe42{3 1 1 2 1 3} + \frac14 \Fe42{3 1 1 2 1 2} + \frac12 \Fe42{3 1 1 1 1 3} + \frac14 \Fe42{3 1 1 1 1 2} + 0.016\times \pi^\sigma\Bigg(\Fe20{3}\Bigg)  
  \\
  & 
  0 = \Fe30{2 2 3} = 
\Fe30{2 3 3}
=
\Fe30{3 3 3}
=
\Fe40{2 2 2 2 2 2},
\end{cases}
\]
where $\pi^\sigma$ is the upward operator, see Razborov~\cite{Razborov}.
The upper bound on the program $(Ph)$ is 0.
The calculation is provided in a similar form as for Lemma~\ref{lemma:fa}.
\end{proof}

By adding additional 215 cuts to the program $(Ph)$, it is possible to improve the bound in Theorem~\ref{lemma:fah} from $n^2/18 - 0.008h$
to
$n^2/18 - 0.025h$. 
We have not tried to push the bound further.

Both calculations are available at \oururl.

\section*{Appendix B: The cuts and formulation for $K_6$}\label{sec:K6cuts}

Here we describe the flag algebra calculation showing that every $K_6$-free graph with sublinear independence number can be made bipartite by removing $(1/12+o(1))n^2$ edges; this, together with the construction  gives $d_2(K_6)=1/6$.
As in Section~\ref{sec:prelim}, by the Regularity Lemma and the argument of Erd\H{o}s, Hajnal, S\'os and Szemer\'edi, it suffices to consider weighted graphs, where every edge is light (weight $1/2$, depicted \Fe202) or heavy (weight $1$, depicted \Fe203). Now the weighted graph is $K_5$-free, and every $K_4$ contains only light edges, and there is no triangle consisting of only heavy edges. These are modeled by forbidding the induced flags
\[
\Fe30{3 3 3}, \qquad \Fe40{2 2 2 2 2 2} \text{ with at least one edge replaced by } \Fe203, \qquad \Fe50{2 2 2 2 2 2 2 2 2 2}.
\]
We show that no such weighted graph has all cuts of weight at least $(1/6+\delta)\binom{n}{2}$ for every $\delta>0$.
The extremal construction is the balanced blow-up of the octahedron $K_{2,2,2}$ on vertices $0,\dots,5$ with non-edges $01$, $23$, $45$, light triangles $024$ and $135$ and heavy $6$-cycle $034125$; every cut below has weight exactly $\frac16\binom n2$ in it for every position of the root vertices.

\medskip
\noindent\textbf{General form of the cuts.}
Each cut is described by a rooted type $\sigma$ (a vertex, or a pair $u,v$ of vertices spanning a non-edge, a light edge or a heavy edge) and a rule sending every remaining vertex $x$ to part $A$, to part $B$, or to a uniformly random part (independently for different vertices), denoted by $\mathrm{A}$, $\mathrm{B}$ and $\mathrm{R}$.
The rule depends only on the \emph{pattern} of $x$, i.e., on the relation of $x$ to the root vertices. We write $0$, $\ell$, $h$ if $x$ is non-adjacent, joined by a light edge, or joined by a heavy edge, respectively; for two roots, the pattern $ab$ means that $x$ has relation $a$ to $u$ and $b$ to $v$.
The weight of the cut is the total weight of edges $xy$ with $x,y$ in the same part, so the cut gives the inequality
\begin{align}\label{eq:K6general}
\frac16 \le \sum_{F} w(xy)\, \mathbb{P}[x,y \text{ in the same part}]\, F,
\end{align}
where the sum is over $\sigma$-flags $F$ on $|\sigma|+2$ vertices with unlabeled vertices $x,y$, $w(xy)\in\{\frac12,1\}$ is the weight of the edge $xy$ in $F$, and the probability is $1$ or $0$ if both $x$ and $y$ are assigned deterministically, and $\frac12$ if at least one of them is assigned randomly.
The root vertices themselves are irrelevant in the limit.

\medskip
\noindent\textbf{Cut 1 (random cut).} Partition all vertices randomly.
\begin{align}
\frac16 \le \frac12\left(\frac12 \Fuu2 + \Fuu3\right)
\end{align}

\noindent\textbf{Cut 2 (neighborhood of a vertex).} Fix a vertex $v$; let $A$ be the non-neighbors of $v$ and $B=N(v)$ (light and heavy neighbors), i.e., $0\mapsto\mathrm{A}$, $\ell\mapsto\mathrm{B}$, $h\mapsto\mathrm{B}$.
\begin{align}\label{eq:K6cut2}
\frac16 \le\ & \frac{1}{2}\Fe{3}{1}{1 1 2} + \Fe{3}{1}{1 1 3} + \frac{1}{2}\Fe{3}{1}{2 2 2} + \Fe{3}{1}{2 2 3} + \frac{1}{2}\Fe{3}{1}{2 3 2} + \Fe{3}{1}{2 3 3} + \frac{1}{2}\Fe{3}{1}{3 3 2}
\end{align}

\noindent\textbf{Cut 3 (non-edge).} Fix a non-adjacent pair $u,v$. Common non-neighbors go to $A$, vertices adjacent to both $u$ and $v$ go to $B$, and vertices adjacent to exactly one of them are placed randomly.
\begin{center}\small
\begin{tabular}{c|ccccccccc}
cut & $00$ & $0\ell $ & $0h$ & $\ell 0$ & $\ell \ell $ & $\ell h$ & $h0$ & $h\ell $ & $hh$ \\ \hline
3 & A & R & R & R & B & B & R & B & B \\
\end{tabular}
\end{center}

\noindent\textbf{Cuts 4--17 (light edge).} Fix a light edge $uv$. The following table gives the part for each pattern; the inequality is then given by \eqref{eq:K6general}.
\begin{center}\small
\begin{tabular}{c|ccccccccc}
cut & $00$ & $0\ell $ & $0h$ & $\ell 0$ & $\ell \ell $ & $\ell h$ & $h0$ & $h\ell $ & $hh$ \\ \hline
4 & R & A & A & A & B & R & A & R & B \\
5 & R & A & A & A & B & R & B & R & B \\
6 & R & A & A & A & B & R & R & R & B \\
7 & R & A & A & B & A & R & B & R & A \\
8 & R & A & A & B & A & R & B & R & B \\
9 & R & A & A & B & A & R & B & R & R \\
10 & R & A & A & B & B & R & A & R & B \\
11 & R & A & A & B & B & R & R & R & B \\
12 & R & A & A & B & R & R & B & R & A \\
13 & R & A & A & B & R & R & B & R & R \\
14 & R & A & A & R & B & R & A & R & B \\
15 & R & A & A & R & B & R & B & R & B \\
16 & R & A & A & R & B & R & R & R & B \\
17 & R & A & B & A & A & R & B & R & B \\
\end{tabular}
\end{center}

\noindent\textbf{Cuts 18--32 (heavy edge).} Fix a heavy edge $uv$. The pattern $hh$ does not occur since there is no heavy triangle.
\begin{center}\small
\begin{tabular}{c|cccccccc}
cut & $00$ & $0\ell $ & $0h$ & $\ell 0$ & $\ell \ell $ & $\ell h$ & $h0$ & $h\ell $ \\ \hline
18 & R & A & A & A & R & B & A & B \\
19 & R & A & A & A & R & B & B & B \\
20 & R & A & A & A & R & B & R & B \\
21 & R & A & A & B & R & A & B & A \\
22 & R & A & A & B & R & A & B & B \\
23 & R & A & A & B & R & A & B & R \\
24 & R & A & A & B & R & B & A & B \\
25 & R & A & A & B & R & B & B & A \\
26 & R & A & A & B & R & B & B & R \\
27 & R & A & A & B & R & B & R & B \\
28 & R & A & A & B & R & R & B & R \\
29 & R & A & A & R & R & B & A & B \\
30 & R & A & A & R & R & B & B & B \\
31 & R & A & A & R & R & B & R & B \\
32 & R & A & B & B & R & B & A & A \\
\end{tabular}
\end{center}

Every cut rooted at a pair $u,v$ is used together with its mirror image obtained by exchanging $u$ and $v$ (i.e., exchanging the patterns $ab$ and $ba$). For Cuts 3, 4, 13, 17, 18, 22, 25, 28 and 32 the rule is symmetric in $u$ and $v$, so the mirror image is the same cut; this gives $53$ inequalities in total.
% I think the complete 
%The complete list of the inequalities written as combinations of flags is available in the supplementary file \texttt{K6\_cuts.pdf}.
The pattern $00$ is always placed randomly; in the light-edge cuts also $\ell h$ and $h\ell$ are placed randomly, because these patterns do not appear in the octahedron construction.

\medskip
\noindent\textbf{The calculation.}
Unlike in the proof of Lemma~\ref{lemma:fa}, we cannot show that the density of heavy edges is zero: the extremal octahedron construction has heavy edges of density $1/3$.
Moreover, the random cut (Cut 1) has value $1/4>1/6$ on the construction, so it cannot be used to normalize a sharp certificate.
Instead, the objective is the averaged Cut 2, which is tight for every choice of the root vertex in the construction.
Averaging the right-hand side of \eqref{eq:K6cut2} over the root vertex gives
\begin{align}\label{eq:K6obj}
\left\llbracket\, \mathrm{Cut\ 2}\, \right\rrbracket = \frac{1}{6}\Fe30{1 1 2} + \frac{1}{3}\Fe30{1 1 3} + \frac{1}{2}\Fe30{2 2 2} + \frac{2}{3}\Fe30{2 2 3} + \frac{5}{6}\Fe30{2 3 3}.
\end{align}
The entire program is the following:
\[
(P_6)
\begin{cases}
    \text{minimize} & \displaystyle\frac16 - \left\llbracket\, \mathrm{Cut\ 2}\, \right\rrbracket \\[8pt]
    \text{subject to} & \displaystyle\frac16 \le \text{right-hand side of \eqref{eq:K6general}} \text{ for Cuts 1--32 (both orders of the roots)},\\[8pt]
    & 0 = \Fe30{3 3 3} = \Fe50{2 2 2 2 2 2 2 2 2 2} = \Fe40{2 2 2 2 2 2} \text{ with any edges replaced by } \Fe203.
\end{cases}
\]
Flag algebra calculation on $6$ vertices, using the cut inequalities multiplied by nonnegative combinations of flags of their types and squares of flags of $54$ types, shows that the minimum of $(P_6)$ is $0$.
More precisely, the certificate is an identity on flags with $6$ vertices
\[
\frac16 - \left\llbracket\, \mathrm{Cut\ 2}\, \right\rrbracket = \sum \text{squares} + \sum_j \lambda_j \left\llbracket F_j \cdot \left(\mathrm{Cut}_j - \frac16\right)\right\rrbracket + \sum_{F} s_F F,
\]
where $\lambda_j\ge 0$, each $F_j$ is a flag of the type of $\mathrm{Cut}_j$, $\mathrm{Cut}_j$ denotes the right-hand side of the corresponding cut inequality, and $s_F\ge 0$.
If every cut had weight at least $(\frac16+\delta)\binom n2$, then the right-hand side would be nonnegative, while $\llbracket \mathrm{Cut\ 2}\rrbracket \ge \frac16+\delta$ makes the left-hand side at most $-\delta$, a contradiction.
The bound is sharp, as all the cuts are tight on the octahedron construction; the certificate was rounded to exact rational numbers, with all slacks $s_F$ vanishing exactly on the $47$ flags on $6$ vertices that appear in the octahedron construction.

\section*{Appendix C: The cuts used for $K_7$}\label{sec:K7cuts}

Here we describe the flag algebra calculation showing that every $K_7$-free graph with sublinear independence number can be made bipartite by removing at most $(1/9+o(1))n^2$ edges.
The complete $3$-partite graph with parts of size $n/3$, each part containing a triangle-free graph with independence number $o(n)$, is $K_7$-free, and every cut of it leaves at least $(1/9+o(1))n^2$ edges inside the parts. Hence $d_2(K_7)=2/9$.

By the Regularity Lemma and the argument of Erd\H{o}s, Hajnal, S\'os and Szemer\'edi, it suffices to consider weighted graphs where every edge is light (weight $1/2$, depicted \Fe202) or heavy (weight $1$, depicted \Fe203).
A cluster can host at most two vertices of a clique, and two clusters can both host two vertices only if the pair between them is heavy; so every clique $S$ of the reduced graph satisfies $|S|+\max(1,\omega_h(S))\le 6$, where $\omega_h(S)$ is the clique number of the heavy edges inside $S$.
Equivalently, there is no $K_6$, no $K_5$ containing a heavy edge, and no $K_4$ containing a heavy triangle.
In other words, there is no
\[
\Fe40{3 3 2 3 2 2}, \qquad \Fe50{3 2 2 2 2 2 2 2 2 2}, \qquad \Fe{6}{0}{2 2 2 2 2 2 2 2 2 2 2 2 2 2 2}
\]
together with all other colourings of $K_4$ containing a heavy triangle and of $K_5$ with a heavy edge.

It gives total
$18$ forbidden configurations, flags, up to isomorphism, as we have to include those minimal cliques as well, which contains more heavy edges.
We show that no such weighted graph has all cuts of weight at least $(2/9+\delta)\binom{n}{2}$ for any $\delta>0$.
The extremal construction is the balanced blow-up of a heavy triangle.

\medskip
\noindent\textbf{General form of the cuts.}
Each cut is described by a rooted type $\sigma$ (no root, a vertex, or a pair $u,v$ of vertices spanning a non-edge, a light edge or a heavy edge) and a rule sending every remaining vertex $x$ to part $A$, to part $B$, or to a uniformly random part (independently for different vertices), denoted by $\mathrm{A}$, $\mathrm{B}$ and $\mathrm{R}$.
The rule depends only on the \emph{pattern} of $x$, i.e., on the relation of $x$ to the root vertices. We write $0$, $\ell$, $h$ if $x$ is non-adjacent, joined by a light edge, or joined by a heavy edge, respectively; for two roots, the pattern $ab$ means that $x$ has relation $a$ to $u$ and $b$ to $v$.
The cut gives the inequality
\begin{align}\label{eq:K7general}
\frac29 \le \sum_{F} w(xy)\, \mathbb{P}[x,y \text{ in the same part}]\, F,
\end{align}
where the sum is over $\sigma$-flags $F$ on $|\sigma|+2$ vertices with unlabeled vertices $x,y$, $w(xy)\in\{\frac12,1\}$ is the weight of the edge $xy$ in $F$, and the probability is $1$ or $0$ if both $x$ and $y$ are assigned deterministically, and $\frac12$ if at least one of them is assigned randomly.
In contrast to the case of $K_6$, a triangle with heavy edges is allowed, so all nine patterns occur for every pair of roots.

\medskip
\noindent\textbf{Cut 1 (random cut).} Partition all vertices randomly.
\begin{align}
\frac29 \le \frac12\left(\frac12 \Fuu2 + \Fuu3\right)
\end{align}

\noindent\textbf{Cuts 2--15 (one root vertex $v$).}
\begin{center}\small
\begin{tabular}{c|ccc}
cut & $0$ & $\ell$ & $h$ \\ \hline
2 & A & A & A \\
3 & A & A & B \\
4 & A & A & R \\
5 & A & B & A \\
6 & A & B & B \\
7 & A & B & R \\
8 & A & R & A \\
9 & A & R & B \\
10 & A & R & R \\
11 & R & A & A \\
12 & R & A & B \\
13 & R & A & R \\
14 & R & R & A \\
15 & R & R & R \\
\end{tabular}
\end{center}

For example, Cut 9 puts the non-neighbors of $v$ into $A$, the heavy neighbors into $B$ and the light neighbors randomly; it is the inequality
\begin{align}\label{eq:K7cut9}
\frac29 \le\ & \frac{1}{2}\Fe{3}{1}{1 1 2} + \Fe{3}{1}{1 1 3} + \frac{1}{4}\Fe{3}{1}{1 2 2} + \frac{1}{2}\Fe{3}{1}{1 2 3} + \frac{1}{4}\Fe{3}{1}{2 2 2} + \frac{1}{2}\Fe{3}{1}{2 2 3} \\ \nonumber
 &+ \frac{1}{4}\Fe{3}{1}{2 3 2} + \frac{1}{2}\Fe{3}{1}{2 3 3} + \frac{1}{2}\Fe{3}{1}{3 3 2} + \Fe{3}{1}{3 3 3}
\end{align}

\noindent\textbf{Cut 16 (non-edge $uv$).}
\begin{center}\small
\begin{tabular}{c|ccccccccc}
cut & $00$ & $0\ell $ & $0h$ & $\ell 0$ & $\ell \ell $ & $\ell h$ & $h0$ & $h\ell $ & $hh$ \\ \hline
16 & A & R & R & R & R & R & R & R & B \\
\end{tabular}
\end{center}

\noindent\textbf{Cuts 24, 25, 28 (light edge $uv$).}
\begin{center}\small
\begin{tabular}{c|ccccccccc}
cut & $00$ & $0\ell $ & $0h$ & $\ell 0$ & $\ell \ell $ & $\ell h$ & $h0$ & $h\ell $ & $hh$ \\ \hline
24 & A & A & A & A & B & A & A & A & B \\
25 & A & A & A & A & B & A & R & A & B \\
28 & A & A & A & B & A & A & B & A & B \\
\end{tabular}
\end{center}

\noindent\textbf{Cuts 17--23, 26, 27 (heavy edge $uv$).}
\begin{center}\small
\begin{tabular}{c|ccccccccc}
cut & $00$ & $0\ell $ & $0h$ & $\ell 0$ & $\ell \ell $ & $\ell h$ & $h0$ & $h\ell $ & $hh$ \\ \hline
17 & R & R & A & R & R & R & A & R & B \\
18 & R & R & A & R & R & R & B & R & A \\
19 & R & R & A & R & R & R & B & R & B \\
20 & R & R & A & R & R & R & B & R & R \\
21 & R & R & A & R & R & R & R & R & B \\
22 & R & R & R & R & R & R & A & R & B \\
23 & A & A & B & B & A & B & A & A & A \\
26 & A & A & B & R & A & B & A & A & A \\
27 & A & A & A & B & A & A & B & B & A \\
\end{tabular}
\end{center}

Cuts 2--15 are all vertex cuts up to exchanging $A$ and $B$, and Cuts 16--22 are the cuts rooted at a pair that are tight on every embedding of the root into the heavy-triangle blow-up (patterns that do not occur in the construction are placed randomly). Cuts 23--28 were added after a first computation: they are the cuts violated most by the optimal solution of the program without them.
Each cut rooted at a pair $u,v$ is used only in the orientation given in the tables.

\medskip
\noindent\textbf{The calculation.}
The random cut has value $\frac13$ on the heavy-triangle blow-up, so, as for $K_6$, we use as objective an averaged rooted cut that is tight in the construction, namely Cut 9.
Averaging the right-hand side of \eqref{eq:K7cut9} over the root vertex gives
\begin{align}\label{eq:K7obj}
\left\llbracket\, \mathrm{Cut\ 9}\, \right\rrbracket = \frac{1}{6}\Fe30{1 1 2} + \frac{1}{3}\Fe30{1 1 3} + \frac{1}{6}\Fe30{1 2 2} + \frac{1}{6}\Fe30{1 2 3} + \frac{1}{4}\Fe30{2 2 2} + \frac{1}{3}\Fe30{2 2 3} + \frac{1}{2}\Fe30{2 3 3} + \Fe30{3 3 3}.
\end{align}
The entire program is the following:
\[
(P_7)
\begin{cases}
    \text{minimize} & \displaystyle\frac29 - \left\llbracket\, \mathrm{Cut\ 9}\, \right\rrbracket \\[8pt]
    \text{subject to} & \displaystyle\frac29 \le \text{right-hand side of \eqref{eq:K7general}} \text{ for Cuts 1--28},\\[8pt]
    & \text{no } K_6, \text{ no } K_5 \text{ with a heavy edge, no } K_4 \text{ with a heavy triangle}.
\end{cases}
\]
Flag algebra calculation on $6$ vertices, using the cut inequalities multiplied by nonnegative combinations of flags of their types and squares of flags of $66$ types, shows that the minimum of $(P_7)$ is $0$.
More precisely, the certificate is an identity on flags with $6$ vertices
\[
\frac29 - \left\llbracket\, \mathrm{Cut\ 9}\, \right\rrbracket = \sum \text{squares} + \sum_j \lambda_j \left\llbracket F_j \cdot \left(\mathrm{Cut}_j - \frac29\right)\right\rrbracket + \sum_{F} s_F F,
\]
where $\lambda_j\ge 0$, each $F_j$ is a flag of the type of $\mathrm{Cut}_j$, $\mathrm{Cut}_j$ denotes the right-hand side of the corresponding cut inequality, and $s_F\ge 0$.
If every cut had weight at least $(\frac29+\delta)\binom n2$, then the right-hand side would be nonnegative, while $\llbracket \mathrm{Cut\ 9}\rrbracket \ge \frac29+\delta$ makes the left-hand side at most $-\delta$, a contradiction.
The certificate was rounded to exact rational numbers; the slacks $s_F$ vanish exactly on the $7$ flags on $6$ vertices that appear in the heavy-triangle blow-up, and all other slacks are at least $10^{-3}$.

\section*{Appendix D: A combinatorial proof of Theorem~\ref{k5}}

As stated in Section~\ref{sec: intro}, Theorem~\ref{wk5up} implies Theorem~\ref{k5}. We give a combinatorial proof of Theorem~\ref{wk5up} here.

Let $G$ be as in Theorem~\ref{wk5up}. We use the notation from Section~\ref{sec: def}. 

%\newpage

For $S\subseteq V(G)$, let $e^w(S)$
be the total weight of the edges inside $S$. Recall that $B(G)$ is the total weight of the edges of the largest bipartite subgraph of $G$.
The bipartite subgraph with parts $S$ and $V(G)\setminus S$ has total edge weight $\sum_{x\in S}d^w_x-2e^w(S)$, so for every $S\subseteq V(G)$ we have
\begin{equation}\label{eq:app-c-cut-2}
    e^w-B(G)\leq e^w-\sum_{x\in S}d^w_x+2e^w(S).
\end{equation}

The proof combines two upper bounds on $e^w-B(G)$.
The first comes from averaging all vertex cuts $(N(v), V(G)\setminus N(v))$.
The second comes from averaging the following slightly more complicated cuts. 
Start by choosing a vertex $v$. Denote by $A_v$ the set of heavy neighbors of $v$, and by $B_v$ the set of light neighbors of $v$.
Let $u$ be a vertex of maximum degree in $G[B_v]$, and denote by $X_v$ its light neighborhood in $B_v$, and set $Y_v:= B_v-X_v$.  
Let  $I$ be either $A_v\cup X_v$ or $A_v\cup Y_v$, with probability $1/2$ each,
see Figure~\ref{fig:cut2} for an illustration.

%We use two independent choices of $I$ to sample the vertices of a set $U$.
We repeat this procedure independently twice, starting with independent choices $v_1$ and $v_2$, to obtain $I_1$ and $I_2$. We then sample two sets $S_1\subseteq I_1$ and $S_2\subseteq I_2$ independently, and take their union $U=S_1\cup S_2$ as one side of a cut.
Averaging these two bounds will reduce the proof to an inequality in one variable.
This step is the new ingredient, suggested by AI. While, in flag algebras, sometimes random cuts are showing up, for this problem, our flag algebra proof did not include such cuts.

\tikzset{
  vtx/.style={circle,fill=black,inner sep=1.6pt,outer sep=0pt},
  heavy/.style={line width=1.2pt},
  light/.style={line width=0.8pt,dashed},
  setA/.style={draw=green!50!black,fill=green!8,thick},
  setX/.style={draw=blue!70!black,fill=blue!8,thick},
  setY/.style={draw=orange!80!black,fill=orange!10,thick},
  setB/.style={draw=black!60,fill=white,thick,rounded corners=12pt},
}

\begin{figure}
\begin{center}
\begin{tikzpicture}
  % v and its edges; the light edges from v end at the boundary of B_v
  \node[vtx,label=above:$v$] (v) at (0,0) {};
  \draw[setA] (-2.8,0) ellipse (0.55 and 1.95);
  \node[green!40!black] at (-2.8,-2.4) {$A_v$};
  \foreach \i/\y in {1/1.2,2/0.4,3/-0.4,4/-1.2} {
    \node[vtx] (a\i) at (-2.8,\y) {};
    \draw[heavy,black!60] (v) -- (a\i);
  }
  \foreach \y in {1.6,0.9,0.2,-0.5,-1.2,-1.9}
    \draw[light,black!35] (v) -- (1.9,\y);

  % B_v = X_v \cup Y_v (drawn opaque)
  \draw[setB] (1.9,-3.05) rectangle (6.9,2.35);
  \node[anchor=north,black!70] at (4.4,-3.05) {$B_v=X_v\cup Y_v$};
  \draw[setX] (4.4,1.4) ellipse (2.2 and 0.55);
  \draw[setY] (4.4,-1.85) ellipse (2.2 and 0.95);
  \node[blue!70!black,anchor=north west] at (6.15,0.97) {$X_v$};
  \node[orange!80!black,anchor=north west] at (6.24,-2.38) {$Y_v$};

  \foreach \i/\x in {1/3.0,2/3.9,3/4.9,4/5.8}
    \node[vtx] (x\i) at (\x,1.4) {};
  \node[vtx] (u)  at (4.4,-1.15) {};
  \node[font=\small,anchor=north west,inner sep=1pt] at (u.south east) {$u$};
  \node[vtx] (y1) at (3.0,-2.0) {};
  \node[vtx] (y2) at (4.6,-2.45) {};
  \node[vtx] (y3) at (5.85,-1.95) {};

  % edges inside B_v (all light); X_v = N(u) \cap B_v
  \foreach \i in {1,...,4} \draw[light,blue!70!black] (u) -- (x\i);
  \draw[light] (y1) -- (x1);
  \draw[light] (y3) -- (x3);
  \draw[light] (y3) -- (x4);
  % edges inside Y_v: at most delta_v of them
  \draw[light,red!80!black,line width=1.1pt] (y1) -- (y2);
  \draw[light,red!80!black,line width=1.1pt] (y2) -- (y3);

  % legend
  \begin{scope}[shift={(-3.4,-4.0)},font=\small]
    \draw[heavy] (0,0) -- (0.8,0) node[right] {heavy edge,};
    \draw[light] (3.2,0) -- (4.0,0) node[right] {light edge,};
    \draw[light,blue!70!black] (6.3,0) -- (7.1,0) node[right,black] {edge at $u$,};
    \draw[light,red!80!black,line width=1.1pt] (9.2,0) -- (10.0,0) node[right,black] {edge inside $Y_v$.};
  \end{scope}
\end{tikzpicture}
\end{center}
\caption{$u$ has maximum degree in $G[B_v]$, $X_v=N(u)\cap B_v$ is independent, $Y_v=B_v\setminus X_v$.}\label{fig:cut2}
\end{figure}

To estimate the vertex cuts, note that for every vertex $v$, the graph $G[B_v]$ on its light neighborhood is triangle-free, and every edge in $G[N(v)]$ is a light edge inside $B_v$. 
Recall that $e^w_v$ is the weight of the edges in $N(v)$, 
$b_v$ is the number of light edges incident to $v$, $d_v = b_v + a_v$ and $d_v^w = b_v/2 + a_v$.

First, we estimate the average of vertex cuts $(N(v), V(G)\setminus N(v))$.
For every vertex $v$, the bipartite subgraph with parts $N(v)$ and
$V(G)\setminus N(v)$ has total edge weight $\sum_{x\in N(v)}d^w_x-2e^w_v$. Since each vertex $x$ belongs to exactly $d_x$
neighborhoods $N(v)$, averaging over all vertices gives
\begin{equation}\label{eq:app-c-lb-b(G)}
    B(G)\geq \frac{1}{n}\sum_{v\in V(G)}
        \left(\sum_{x\in N(v)}d^w_x-2e^w_v\right)
%      =\frac{1}{n}\sum_{v\in V(G)} d_vd^w_v-\frac{2}{n}\sum_{v\in V(G)} e^w_v.
      =\frac{1}{n}\sum_{v\in V(G)}\left( d_vd^w_v-2e^w_v \right).
\end{equation}
Let
\[
    \delta_v\coloneqq \frac{b_v^2}{4}-2e^w_v,\qquad
    \Delta\coloneqq \sum_{v\in V(G)}\delta_v.
\]
Using
%$d_vd^w_v-2e^w_v=(d^w_v)^2+a_vb_v/2+\delta_v$
\[
d_vd^w_v-2e^w_v=
(d^w_v)^2+\frac12b_v\left(\frac12b_v+a_v\right)-2e^w_v =
(d^w_v)^2+\frac{1}{2}a_vb_v+\delta_v,
\] by~\eqref{eq:app-c-lb-b(G)} we obtain
\begin{equation}\label{eq:app-c-neighborhood}
\begin{aligned}
    e^w-B(G)&\leq e^w-\frac1n\sum_{v\in V(G)}\left(d_vd^w_v-2e^w_v\right)\\
    &=e^w-\frac1n\sum_{v\in V(G)}(d^w_v)^2
        -\frac1{2n}\sum_{v\in V(G)}a_vb_v-\frac{\Delta}{n}.
\end{aligned}
\end{equation}

To construct the second cut, we first find a partition $(X_v,Y_v)$ of each $B_v$ with at most $\delta_v$ edges inside the two parts as follows. If $B_v$ is empty, we take $X_v=Y_v=\emptyset$.
Otherwise, let $X_v$ be the neighborhood of a vertex of maximum degree $d$ in $G[B_v]$. 
\begin{claim}\label{Yv}
    The number of edges inside $Y_v$ is at most $\delta_v$.
\end{claim}
\begin{proof}
  Note that $X_v$ is an independent set, as $G[B_v]$ is triangle-free.
  Hence, for
$Y_v=B_v\setminus X_v$ we have
\[
    e(G[B_v])+e(G[Y_v])=\sum_{x\in Y_v}d_{G[B_v]}(x)
    \leq d(b_v-d)\leq \frac{b_v^2}{4}.
\]
Since $e(G[B_v])=2e^w_v$, and $X_v$ is an independent set, we conclude that 
the number of edges inside $Y_v$ is at most 
 $b_v^2/4-2e^w_v=\delta_v$.
\end{proof}

Choose $v$ uniformly at random, and let $I$ be the union of $A_v$ and one of $X_v$ and $Y_v$, chosen with equal probability $1/2$.
For a given $v$, a vertex $x$ belongs to $I$ with probability $1$ if $v\in A_x$, with probability $1/2$ if $v\in B_x$, and with probability $0$ otherwise.
Averaging over all vertices $v$ gives 
\[
\mathbb{P}[x\in I]=\frac1n\left(a_x+\frac{b_x}{2}\right)=\frac{d^w_x}{n}.
\]

Since no edge of $G[N(v)]$ meets $A_v$, the edges within the chosen part (i.e., within $Y_v$, if it was chosen) are the only edges that contribute to $e^w(I)$. 
By Claim~\ref{Yv},
there are at most $\delta_v$ edges within  $Y_v$, and they lie in $G[I]$ with probability $1/2$,  and each such edge has weight $1/2$. %$\mathbb{E}[e^w(I)\mid v]\leq\delta_v/4$
Averaging over all vertices $v$ gives
\[
\mathbb{E}[e^w(I)]\leq\frac1n\sum_{v\in V(G)}\frac{\delta_v}{4}=\frac{\Delta}{4n}.
\]

We now take the union of two samples to increase the sum of weighted degrees in~\eqref{eq:app-c-cut-2} as follows.
Let $S\subseteq I$ be a random subset of $I$, with each vertex $x\in I$ belonging to $S$ with probability $\min\{1,n/(3d^w_x)\}$ independently.
We made this choice, as we want to minimize the right hand side of~\eqref{eq:app-c-union}.

Write $p_x=\mathbb{P}[x\in S]$ for  $x\in V(G)$. Note that
\begin{align}
    p_x=\mathbb{P}[x\in S \mid x\in I]\cdot \mathbb{P}[x\in I] = \min\ \left\{1,\frac{n}{3d^w_x}\right\} \cdot \frac{d^w_x}{n} = \min\  
    \left\{\frac{d^w_x}{n},\frac{1}{3}\right\}.
\end{align}

Let $S_1$ and $S_2$ be two independent copies of $S$, obtained by independently repeating the entire procedure, including choosing the starting vertices $v_1$ and $v_2$. Let $U=S_1\cup S_2$.
An edge $xy$ inside $U$ either lies inside one of $S_1$ and $S_2$ or has one
endpoint in each. Denote the weight of $xy$ by $w(xy)$, then
\begin{align*}
    \mathbb{E}[e^w(U)]
      &\leq 2\mathbb{E}[e^w(S)]+2\sum_{xy\in E(G)}w(xy)p_xp_y\\
      &\leq \frac{\Delta}{2n}
           +\sum_{xy\in E(G)}w(xy)(p_x^2+p_y^2)
       =\frac{\Delta}{2n}+\sum_{x\in V(G)} d^w_x p_x^2.
\end{align*}
Since $\mathbb{P}[x\in U]=1-\mathbb{P}[x\notin S_1]\cdot \mathbb{P}[x\notin S_2]=1-(1-p_x)^2=2p_x-p_x^2$, taking expectations in
\eqref{eq:app-c-cut-2} with $S=U$ gives
\begin{equation}\label{eq:app-c-union}
\begin{aligned}
    e^w-B(G)&\leq e^w-\sum_{x\in V(G)}d^w_x(2p_x-p_x^2)+2\left(\frac{\Delta}{2n}+\sum_{x\in V(G)}d^w_xp_x^2\right)\\
    &=e^w-\sum_{x\in V(G)}d^w_x(2p_x-3p_x^2)+\frac{\Delta}{n}.
\end{aligned}
\end{equation}
The choice $p_x=\min\{d^w_x/n,1/3\}$ maximizes $2p_x-3p_x^2$ in the interval 
$0\leq p_x\leq \mathbb{P}[x\in I]=d^w_x/n$ and thus minimizes the right-hand side
of~\eqref{eq:app-c-union}.

Adding up~\eqref{eq:app-c-neighborhood} and \eqref{eq:app-c-union}, and dividing by $2$  gives
\[
    e^w-B(G)\leq e^w-\frac{1}{2n}\sum_{v\in V(G)}(d^w_v)^2-\frac12\sum_{v\in V(G)}d^w_v(2p_v-3p_v^2)-\frac{1}{4n}\sum_{v\in V(G)}a_vb_v.
\]
Let $z_v=d^w_v/n$ and $p(z)=\min\{z,1/3\}$, then $p_v=p(z_v)$ by definition. Let
\[
    f(z)=\frac{z}{2}-\frac{z^2}{2}-zp(z)+\frac{3}{2}zp(z)^2.
\]
Using $e^w=\frac{1}{2}\sum_{v\in V(G)}d^w_v$, we obtain
\[
    e^w-B(G)\leq n\sum_{v\in V(G)}f\left(\frac{d^w_v}{n}\right)-\frac{1}{4n}\sum_{v\in V(G)}a_vb_v.
\]
The last term is non-positive, so it suffices to show that $f(z)\leq 1/18$ for every $z\geq 0$.
Indeed,
\[
    \frac{1}{18}-f(z)=
    \begin{cases}
        \dfrac{(1-3z)^3}{18},\quad\quad\quad & \text{if} \quad 0\leq z\leq 1/3,\\[12pt]
        \dfrac{1}{2}\left(z-\dfrac{1}{3}\right)^2,\quad\quad\quad  & \text{if} \quad  z\geq 1/3,
    \end{cases}
\]
giving that $e^w-B(G)\leq n^2/18$.

\end{document}